\documentclass[11pt]{article}
\title
{The phase transition in site percolation on pseudo-random graphs}
\author{Michael Krivelevich
\thanks{School of Mathematical Sciences, Raymond and Beverly
Sackler Faculty of Exact Sciences, Tel Aviv University, Tel Aviv,
6997801, Israel. Email: krivelev@post.tau.ac.il. Research supported in
part by a USA-Israel BSF grant and by a grant from the Israel
Science Foundation.}
}

\usepackage{amsmath,amsthm, amssymb,latexsym, amsfonts}

\begin{document}
\bibliographystyle{plain}
\maketitle
\newtheorem{thm}{Theorem}
\newtheorem{defin}{Definition}
\newtheorem{lemma}{Lemma}[section]
\newtheorem{corol}[lemma]{Corollary}
\newtheorem{thmtool}{Theorem}[section]
\newtheorem{corollary}[thmtool]{Corollary}
\newtheorem{lem}[thmtool]{Lemma}
\newtheorem{prop}[thmtool]{Proposition}
\newtheorem{clm}[thmtool]{Claim}
\newtheorem{conjecture}{Conjecture}
\newtheorem{problem}{Problem}
\newcommand{\Proof}{\noindent{\bf Proof.}\ \ }
\newcommand{\Remarks}{\noindent{\bf Remarks:}\ \ }
\newcommand{\Remark}{\noindent{\bf Remark:}\ \ }
\newcommand{\whp}{{\bf whp}\ }
\newcommand{\prob}{probability}
\newcommand{\rn}{random}
\newcommand{\rv}{random variable}
\newcommand{\hpg}{hypergraph}
\newcommand{\hpgs}{hypergraphs}
\newcommand{\subhpg}{subhypergraph}
\newcommand{\subhpgs}{subhypergraphs}
\newcommand{\bH}{{\bf H}}
\newcommand{\cH}{{\cal H}}
\newcommand{\cT}{{\cal T}}
\newcommand{\cF}{{\cal F}}
\newcommand{\cD}{{\cal D}}
\newcommand{\cC}{{\cal C}}

\begin{abstract}
We establish the existence of the phase transition in site percolation on pseudo-random $d$-regular graphs. Let $G=(V,E)$ be an $(n,d,\lambda)$-graph, that is, a $d$-regular graph on $n$ vertices in which all eigenvalues of the adjacency matrix, but the first one, are at most $\lambda$ in their absolute values. Form a random subset $R$ of $V$ by putting every vertex $v\in V$ into $R$ independently with probability $p$. Then for any small enough constant $\epsilon>0$, if $p=\frac{1-\epsilon}{d}$, then with high probability all connected components of the subgraph of $G$ induced by $R$ are of size at most logarithmic in $n$, while for $p=\frac{1+\epsilon}{d}$, if the eigenvalue ratio $\lambda/d$ is small enough as a function of $\epsilon$, then typically $R$ contains a connected component of size at least $\frac{\epsilon n}{d}$ and a path of length proportional to $\frac{\epsilon^2n}{d}$.
\end{abstract}

\section{Introduction and main results}

Let $G=(V,E)$ be a $d$-regular graph on $n$ vertices. Form a random vertex subset $R\subseteq V$ by putting every vertex $v\in V$ into $R$ independently with probability $p$. What can be said about the properties of the random subgraph of $G$ induced by $R$? How large a connected component does it typically contain? How long a path can one find with high probability ({\bf whp}) in $G[R]$?

Of course, the above model is nothing else but the {\em site percolation}, sometimes also called the vertex percolation, on $G$. Although it is perhaps somewhat less popular than its sister model of bond (edge) percolation, it has been quite extensively studied for various graphs and probability regimes.

A well tested intuition suggests that interesting things start happening when the expected vertex degree in the so formed random subgraph crosses the value of 1. This should correspond to the vertex probability $p=\frac{1}{d}$. For this regime, we expect the cardinality of $R$ to be about $n/d$, and it is thus natural so scale the obtained structures relative to this size. There are several results of this type, showing the typical emergence of a connected component $C$ whose size is proportional to $n/d$ for several concrete graphs, like the $d$-dimensional cube $Q^d$ (\cite{BKL91}, \cite{R09}), or the $n$-dimensional Hamming torus \cite{S14} (in fact, statements much more accurate than those to be presented here have been obtained for these models). We however aim to obtain a result applicable to a large class of $d$-regular graphs.

Certainly some further assumptions on the ground graph $G$ have to be made if we aim to get a positive result, that is, to claim the typical existence of a large connected component spanned by $R$. Indeed, we can start with the graph $G$ being a collection of vertex disjoint cliques of size $d+1$ --- in which case of course all connected components in $R$ are of size at most $d+1$, much smaller than $n/d$ for small degree $d=d(n)$. Thus, it is natural to impose some restrictions on the edge distribution of $G$.

Here we assume that $G$ is a pseudo-random graph. Informally speaking, a pseudo-random graph is a graph $G=(V,E)$, whose edge distribution resembles closely that of a truly random graph $G(n,p)$ of the same edge density $p=\frac{2|E|}{|V|}$. There are several possible models of pseudo-random graphs commonly used. In this paper we adapt the notion of $(n,d,\lambda)$-graphs. A graph $G$ is an {\em $(n,d,\lambda)$-graph} if $G$ has $n$ vertices, is $d$-regular, and all eigenvalues of the adjacency matrix of $G$, but the first one, are at most $\lambda$ in their absolute values. (We assume that the eigenvalues of (the adjacency matrix of) $G$ are ordered in the non-increasing order $\lambda_1\ge\ldots\ge \lambda_n$. The largest eigenvalue of any $d$-regular graph is easily seen to be $d$, sometimes referred to as the trivial eigenvalue of $G$.) The reader can consult the survey \cite{KS06} for an extensive discussion of this notion.

Results about bond (edge) percolation on $(n,d,\lambda)$-graphs have appeared in \cite{FKM04}, \cite{KS13}. To the best of our knowledge, the present paper is the first one to address the general setting of site percolation on $(n,d,\lambda)$-graphs, or on some other general class of regular pseudo-random graphs.

It is well known that the pseudo-randomness of the edge distribution in $(n,d,\lambda)$-graphs can be controlled through the so called {\em eigenvalue ratio} $\lambda/d$ --- the smaller the ratio is the closer the edge distribution of $G$ approaches that of a random graph with edge probability $p=\frac{d}{n}$. We will state a standard result establishing this connection later (Lemma \ref{eml}).

Equipped with this formalism, we can now state our main results.

\begin{thm}\label{thm1}
Let $\epsilon>0$. Let $G=(V,E)$ be a graph of maximum degree at most $d$ on $n$ vertices. Form a random subset $R\subseteq V$ by including each vertex $v\in V$ in $R$ independently and with probability $p$. If $p=\frac{1-\epsilon}{d}$, then \whp all connected components of the induced subgraph $G[R]$ are of size less than $\frac{4}{\epsilon^2}\ln n$.
\end{thm}

\begin{thm}\label{th2}
For every small enough $\epsilon>0$ there exists $\delta>0$ such that the following is true. Let $G=(V,E)$ be an $(n,d,\lambda)$-graph. Assume that $d=o(n)$ and $\frac{\lambda}{d}<\delta$. Let $p=\frac{1+\epsilon}{d}$. Form a random subset $R\subseteq V$ by including each vertex $v\in V$ in $R$ independently and with probability $p$. Then \whp $R$ contains a path of length at least $\frac{\epsilon^2n}{5d}$ in $G$.
\end{thm}

\begin{thm}\label{th3}
For every small enough $\epsilon>0$ there exists $\delta>0$ such that the following is true. Let $G=(V,E)$ be an $(n,d,\lambda)$-graph. Assume that $d=o(n)$ and $\frac{\lambda}{d}<\delta$. Let $p=\frac{1+\epsilon}{d}$. Form a random subset $R\subseteq V$ by including each vertex $v\in V$ in $R$ independently and with probability $p$. Then \whp the induced subgraph $G[R]$ has a connected component of size at least $\frac{\epsilon n}{d}$.
\end{thm}

Some comments are in order here. First, observe that Theorem \ref{thm1} holds unconditionally, i.e., without any further assumptions on the edge distribution of $G$ --- it applies to {\em every} graph $G$ of maximum degree $d$. This means that if the vertex probability $p=p(d)$ is a notch below the critical value $1/d$, then even for the best ground graphs $G$ the random induced subgraph $G[R]$ typically shatters into relatively small pieces. On the positive side, for $p(d)$ above the critical probability and assuming that $G$ is pseudo-random, $R$ contains typically a component of size linear in $|R|$, and even a path of linear size. This phenomenon can be viewed as the {\em phase transition} in this site percolation model. It is pretty much in line with the familiar situation in the random graph $G(n,p)$.
There, for $p=\frac{1-\epsilon}{n}$ all components are of size $O_{\epsilon}(\log n)$, while for $p=\frac{1+\epsilon}{n}$, there is \whp a connected component of size linear in $n$. This is, in one sentence, the essence of the fundamental discovery of Erd\H{o}s and R\'enyi \cite{ER60}. As for paths, Ajtai, Koml\'os and Szemer\'edi proved some 20 years later \cite{AKS81} that $G(n,p)$ with $p=\frac{1+\epsilon}{n}$ contains \whp a path of length linear in $n$ as well (see \cite{KS13} for a recent simple proof of this fact). Actually, even the order of the dependence on $\epsilon$ in our theorems matches the corresponding results for $G(n,p)$.

Let us say a few general words about the proofs. We use the Depth First Search algorithm (DFS) for all three theorems above. We run the DFS algorithm on our random instance, allowing it to uncover the random set $R$ along the algorithm execution. In this respect, our arguments are somewhat similar to those of \cite{KS13}, with the most substantial difference being that in our setting the algorithm exposes random decisions on the {\em vertices} of $G$, rather than its its edges as in \cite{KS13}. Another key ingredient in the proof is an estimate of the number of non-expanding sets of a given size in an $(n,d,\lambda)$-graph; here we are pretty much inspired by a similar argument from the paper of Alon and R\"odl \cite{AR05}.

The notation we use here is fairly standard. In particular, for a graph $G=(V,E)$ and disjoint vertex subsets $U,W\subset V$, we denote by $N_G(U)$ the external neighborhood of $U$, i.e., $N_G(U)=\{v\in V\setminus U:\mbox{$v$ has a neighbor in $U$}\}$; let also $e_G(U)$, $e_G(U,W)$ denote the number of edges of $G$ spanned by $U$, between $U$ and $W$, resp. For $v\in V$ and $U\subseteq V$, let $d(v,U)$ be the number of neighbors of $v$ in $U$. If the graph $G$ is clear from the context, we often allow ourselves not to put it in the indices of the above notation. We omit rounding signs occasionally for the the sake of clarity of presentation.

\section{Tools}
\subsection{Eigenvalues and edge distribution}
We will apply the following standard estimate (see, e.g., Chapter 9 of \cite{AS}), sometimes called the expander mixing lemma; it postulates that the edge distribution in an $(n,d,\lambda)$-graph $G$ with small eigenvalue ratio $\lambda/d$ is quite close to that of a truly random graph of edge density $d/n$. In fact this will be the only tool about graph eigenvalues used in our proof.
\begin{lemma}\label{eml}
Let $G=(V,E)$ be an $(n,d,\lambda)$-graph. Then for any pair of subsets $B,C\subseteq V$,
\begin{equation}\label{eml1}
\left|e(B,C)-\frac{d}{n}|B||C|\right|\le \lambda \sqrt{|B||C|}\,,
\end{equation}
where $e(B,C)$ denotes the number of ordered pairs $(u,v)$ with $u\in B,v\in C$ such that $(u,v)\in E$.
\end{lemma}

\begin{corol}\label{cor1}
Let $G=(V,E)$ be an $(n,d,\lambda)$-graph and let $\alpha>0$. Let $B\subseteq V$ be a vertex subset of cardinality $|B|\ge \frac{n}{2}$. Define
$$
C=\left\{v\in V: d(v,B)\le (1-\alpha)\frac{|B|d}{n}\right\}\,.
$$
Then $|C|\le \frac{2}{\alpha^2}\left(\frac{\lambda}{d}\right)^2n$.
\end{corol}

\begin{proof}
Observe that by the definition of $C$ we have $e(B,C)\le (1-\alpha)\frac{|B||C|d}{n}$. On the other hand, by (\ref{eml1}) $e(B,C)\ge \frac{d}{n}|B||C|-\lambda\sqrt{|B||C|}$. Comparing we derive:
$$
\frac{d}{n}|B||C|-\lambda\sqrt{|B||C|}\le (1-\alpha)\frac{|B||C|d}{n}\,,
$$
and from here
$$
|C|\le \frac{n^2}{|B|}\cdot\frac{1}{\alpha^2}\left(\frac{\lambda}{d}\right)^2\le \frac{2}{\alpha^2}\left(\frac{\lambda}{d}\right)^2n
$$
(the last inequality is due to the assumption $|B|\ge \frac{n}{2}$), and the claim follows.
\end{proof}

\subsection{Depth first search on random vertex subgraphs}
The {\em Depth First Search}, or DFS for brevity, is a standard graph search algorithm, usually used to uncover the connected components of an input graph $G=(V,E)$. In this paper we use it in a somewhat unusual context, revealing also a random subset $R\subset V$ of the vertex set $V$ as the algorithm proceeds. Here is a brief description of the algorithm. It maintains and updates a partition of $V$ into four sets of vertices, letting
$S$ be the set of vertices whose exploration is complete, $T$ be the
set of unvisited vertices, $U$ be the set of presently processed vertices, where the
vertices of $U$ are kept in a stack (the last in, first out data
structure), and finally $W$ be the set of vertices discovered to fall outside of the random set $R$. It is also assumed that some order $\sigma$ on the
vertices of $G$ is fixed, and the algorithm prioritizes vertices
according to $\sigma$. The algorithm starts with $S=U=W=\emptyset$ and
$T=V$, and runs till $U\cup T=\emptyset$. At each round of the
algorithm, if the set $U$ is non-empty, the algorithm queries $T$
for neighbors of the last vertex $v$ that has been added to $U$,
scanning these neighbors according to $\sigma$. If $v$ has a neighbor $u$  in
$T$, the algorithm flips a coin that comes heads with probability $p$. If the result of this coin flipping is positive, the algorithm deletes $u$ from $T$ and inserts it into $U$; otherwise $u$ moves to $W$.
If $v$ does not have a neighbor in $T$, then $v$ is popped out of $U$
and is moved to $S$. If $U$ is empty, the algorithm chooses the
first vertex $u$ of $T$ according to $\sigma$, deletes it from $T$, flips a coin for $u$ and either pushes it
into $U$ or moves to $W$ based on the result of this coin flipping. Once the algorithm execution is complete, the set $S$ coincides with the random set $R$, while $W$ is its complement $W=V\setminus R$.

Observe that the DFS algorithm starts revealing a connected
component $C$ of the induced subgraph $G[R]$ at the moment the first vertex of $C$ gets into
(empty beforehand) $U$ and completes discovering all of $C$ when $U$
becomes empty again. We call a period of time between two
consecutive emptyings of $U$ an {\em epoch}, each epoch corresponding
to one connected component of $G[R]$.

The following basic properties of the DFS algorithm will be useful to
us:
\begin{itemize}
\item at any stage of the algorithm, it has been revealed already
that the graph $G$ has no edges between the current set $S$ and the
current set $T$, and thus $N_G(S)\subseteq U\cup W$;
\item the set $U$ always spans a path (indeed, when a vertex $u$ is
added to $U$, it happens because $u$ is a neighbor of the last
vertex $v$ in $U$; thus, $u$ augments the path spanned by $U$, of
which $v$ is the last vertex).
\end{itemize}

We will run the DFS on an $n$-vertex input $G$, fixing some
order $\sigma$ on $V(G)$. When
the DFS algorithm is fed with a sequence of i.i.d. Bernoulli($p$)
random variables $\bar{X}=(X_i)_{i=1}^n$, so that is gets its $i$-th
query answered positively if $X_i=1$ and answered negatively
otherwise, the final subset $S$ of the algorithm is distributed {\em exactly} like a random subset $R$, formed by including each vertex of $V$ independently and with probability $p$.
 Thus, studying the structure of $G[R]$ can be
reduced to studying the properties of the random sequence $\bar{X}$ --- a much more accessible task.

\subsection{Concentration of random variables}
As we indicated, our argument allows to study the properties of a random vector $\bar{X}=(X_i)_{i=1}^n$, instead of studying directly the subgraph of $G$, spanned by a random subset $R$. Since for a subset $I\subseteq [n]$, the sum $\sum_{i\in I}X_i$ is distributed binomially with parameters $|I|$ and $p$, we can use standard large deviation estimates for binomial random variables. In particular, we have:
\begin{lemma}\label{lem1}
Let $\epsilon>0$ be a small enough constant. Consider the sequence
$\bar{X}=(X_i)_{i=1}^n$ of i.i.d. Bernoulli random variables with
parameter $p$. Assume $d=o(n)$.
\begin{enumerate}
\item Let $p=\frac{1-\epsilon}{d}$.  Let $k=\frac{4}{\epsilon^2}\ln
n$. Then \whp there is no interval of length $kd$ in $[n]$, in which
at least $k$ of the random variables $X_i$ take value 1.
\item Let $p=\frac{1+\epsilon}{d}$. Then \whp\ $\sum_{i=1}^{\epsilon^3n}X_i\le \frac{2\epsilon^3n}{d}$.
\item Let $p=\frac{1+\epsilon}{d}$. Then \whp\ $\sum_{i=1}^{\epsilon n}X_i\le \frac{2\epsilon n}{d}$.
\item Let $p=\frac{1+\epsilon}{d}$. Then \whp\ for every $\epsilon^3n\le t\le \epsilon n$,
 $\sum_{i=1}^{t} X_i\ge \frac{\left(1+\frac{3\epsilon}{4}\right)t}{d}$.
\end{enumerate}
\end{lemma}

\begin{proof}
\noindent{(1)} For a given interval $I$ of length $kd$ in $[n]$, the
sum $\sum_{i\in I} X_i$ is distributed binomially with parameters
$kd$ and $p$. Applying the standard Chernoff-type bound (see, e.g.,
Theorem A.1.11 of \cite{AS}) to the upper tail of $\textnormal{Bin}(kd,p)$, and
then the union bound, we see that the probability of the existence
of an interval violating the assertion of the lemma is at most
$$
(n-kd+1)Pr[B(kd,p)\ge k]< n\cdot
e^{-\frac{\epsilon^2}{3}(1-\epsilon)k} <n\cdot
e^{-\frac{\epsilon^2(1-\epsilon)}{3}\,\frac{4}{\epsilon^2}\ln n}=
o(1)\,,
$$
for small enough $\epsilon>0$.

\noindent{(2)} This follows by applying Chernoff to the upper tail of $\sum_{i=1}^{\epsilon^3n}X_i\sim \textnormal{Bin}\left(\epsilon^3n,\frac{1+\epsilon}{d}\right)$.

\noindent{(3)} This follows by applying Chernoff to the upper tail of $\sum_{i=1}^{\epsilon n}X_i\sim \textnormal{Bin}\left(\epsilon n,\frac{1+\epsilon}{d}\right)$.

\noindent{(4)} Partition $[\epsilon n]$ into $1/\epsilon^3$ intervals $I_j$ of length $\epsilon^4n$ each. Applying Chernoff to the lower tails of the interval sums $\sum_{i\in I_j}X_i$ and then the union bound, we derive that \whp for all $j$
$$
\sum_{i\in I_j}X_i\ge \frac{\epsilon^4(1+\epsilon)n}{d}-\frac{\epsilon^6n}{d}\,,
$$
say. Assume this to be true. Then for $\epsilon^3n\le t\le \epsilon n$,
$$
\sum_{i=1}^{t}X_i \ge \left\lfloor \frac{t}{\epsilon^4n}\right\rfloor\,\left(\frac{\epsilon^4(1+\epsilon)n}{d}-\frac{\epsilon^6n}{d}\right) \ge \frac{\left(1+\frac{3\epsilon}{4}\right)t}{d}\,.
$$
\end{proof}

\section{Proofs}
\subsection{Proof of Theorem \ref{thm1}}
Assume to the contrary that $R$ contains a connected
component $C$ with at least  $k=\frac{4}{\epsilon^2}\ln n$ vertices.
Let us look at the epoch of the DFS when $C$ was created. Consider
the moment inside this epoch when the algorithm  found the
$k$-th vertex of $C$ and has just moved it to $U$. Denote
$C_0=(S\cup U)\cap C$ at that moment. Then $|C_0|=k$, the subgraph $G[C_0]$ is connected and thus spans at least $k-1$ edges. Notice that
$$
|N_G(C_0)|\le e_G(C_0,V-C_0)=\sum_{v\in C_0}d_G(v)-2e_G(C_0)\le kd-2(k-1)\,.
$$
The algorithm got exactly $k$ positive  answers to its queries
to random variables $X_i$ during the epoch, with each positive
answer being responsible for revealing yet another vertex of $C_0$.
At this moment during the epoch only the vertices in $C_0$ and those neighboring
them in $G$ have been queried, and the number of these vertices is
therefore at most $k+kd-2(k-1)\le kd$. It thus follows that the
sequence $\bar{X}$ contains an interval of length at most $kd$ with
at least $k$ 1's inside --- a contradiction to Property 1 of Lemma
\ref{lem1}.

\subsection{Proofs of Theorems \ref{th2} and \ref{th3}}
The proofs of Theorems \ref{th2} and \ref{th3} are based on the same lemma we present and prove next. Observe that in an $(n,d,\lambda)$-graph $G=(V,E)$, every vertex subset $S$ expands itself outside by at most the factor of $d$. The assumption $\lambda/d\le \delta$ we put on the eigenvalue ratio is fairly mild (indeed, best pseudo-random graphs are known to satisfy $\lambda=\Theta(\sqrt{d})$, see, e.g., \cite{KS06}), and cannot guarantee such expansion for {\em all sets}. The key lemma below asserts that sets $S\subset V$ of relevant size $|S|=\Theta(n/d)$, not expanding themselves by nearly the factor of $d$, are very rare in $G$ even under this weak eigenvalue ratio assumption, and thus  are unlikely to fall into a random subset $R$ of size proportional to $n/d$.

\begin{lemma}\label{key_lemma}
For every $0<\alpha_0<1$, $0<c\le 1/3$ there exists $\delta>0$ such that the following is true. Let $G=(V,E)$ be an $(n,d,\lambda)$-graph.
Assume that $d=o(n)$ and $\frac{\lambda}{d}<\delta$. Let $p\le \frac{2}{d}$. Form a random subset $R\subseteq V$ by including each vertex $v\in V$ in $R$ independently and with probability $p$. Then \whp\ $R$ does not contain a set $S$ with $|S|=m$, $\frac{cn}{d}\le m\le \frac{n}{3d}$, such that $|N_G(S)|< (1-\alpha_0)\left(dm-\frac{d^2m^2}{2n}\right)$.
\end{lemma}

\begin{proof}
A set $S\subseteq V$, $|S|=m$, is called {\em non-expanding} if $|N_G(S)|< (1-\alpha_0)\left(dm-\frac{d^2m^2}{2n}\right)$, and is {\em expanding} otherwise. We estimate from above the number of non-expanding $m$-sets in $G$. Define $0<\alpha<1$ by $1-\alpha_0=\frac{(1-\alpha)^2}{1+\alpha}$. (The function $f(x)=\frac{(1-x)^2}{1+x}$ is monotone decreasing in the interval $[0,1]$ with $f(0)=1$, $f(1)=0$, so $\alpha=\alpha(\alpha_0)$ as above is indeed well defined.) Consider the number of ways to choose a sequence $\tau=(v_1,\ldots,v_m)$ of distinct
vertices of $G$ such that the union $S$ of the vertices in the
sequence forms a non-expanding set. Suppose we have chosen the first
$i-1$ vertices of $\tau$, let $S_{i-1}=\{v_1,\ldots,v_{i-1}\}$, and  $N_{i-1}=N_G(S_{i-1})$. A vertex $v$ is {\em bad} with respect to the
prefix $(v_1,\ldots,v_{i-1})$ if $v$ has at
most $(1-\alpha)\frac{d}{n}(n-(d+1)(i-1))$ neighbors in $V-(S_{i-1}\cup N_{i-1})$, and is {\em good} otherwise.
Each good vertex $v_i$ appended to $S_{i-1}$ increases substantially the external neighborhood of the prefix.
Suppose that $\tau$ has at most $\alpha m$ bad vertices. Then
\begin{eqnarray*}
|N_G(S)|&=&|N_G(S_m)|\ge \sum_{i=\alpha m+1}^{m}(1-\alpha)\frac{d}{n}(n-(d+1)(i-1))-m\\
&\ge& (1-\alpha)^2dm - \frac{(1-\alpha)d(d+1)}{n}\sum_{i=\alpha m+1}^m(i-1)-m\\
&\ge& (1-\alpha)^2dm - \frac{(1-\alpha)^2(1+\alpha)d(d+1)m^2}{2n}-m\\
&\ge& (1-\alpha)^2dm\left(1-\frac{(1+\alpha)dm}{2n}\right)-\frac{(1+\alpha)dm^2}{2n}-m\,.
\end{eqnarray*}
(We subtracted $m$ in the first line above to account for the the $m$ vertices of $S$ itself, not contributing to the external neighborhood of $S$.) We need to verify that $S$ is an expanding set. Observe that $\frac{(1+\alpha)d
m^2}{2n}<\frac{2dm^2}{2n}\le \frac{m}{3}<m$, due to our assumption on $m$. Also,
\begin{gather*}
(1-\alpha)^2dm\left(1-\frac{(1+\alpha)dm}{2n}\right)-(1-\alpha_0)dm\left(1-\frac{dm}{2n}\right)\\
=(1-\alpha)^2dm\left(1-\frac{(1+\alpha)dm}{2n}-\frac{1-\frac{dm}{2n}}{1+\alpha}\right)\\
=\frac{(1-\alpha)^2}{1+\alpha}dm\left(1+\alpha-(1+\alpha)^2\frac{dm}{2n}-1+\frac{dm}{2n}\right)\\
(1-\alpha_0)\alpha dm\left(1-(2+\alpha)\frac{dm}{2n}\right)\,.
\end{gather*}
Since $\frac{dm}{2n}\le \frac{1}{6}$ and $2+\alpha<3$, the expression above is at least $(1-\alpha_0)\alpha dm/2>2m$, for large enough $d$, the latter can be guaranteed by our choice of $\delta$. It follows that in this case the set $S$ is indeed expanding. Hence, in order to produce a non-expanding set $S$, the sequence $\tau$ should contain at least $\alpha m$ bad vertices. Let us zoom in at the $i$-th vertex $v_i$ of $\tau$. Given $v_1,\ldots,v_{i-1}$, the set $S_{i-1}\cup N_{i-1}$ has obviously at most $(i-1)(d+1)<m(d+1)<n/2$ vertices. Then by Corollary \ref{cor1} the number of bad choices for $v_i$ is at most $\frac{2}{\alpha^2}\left(\frac{\lambda}{d}\right)^2n\le \frac{2}{\alpha^2}\delta^2n$. Therefore the number of sequences $\tau$ with at least $\alpha m$ bad vertices is at most
$$
\binom{m}{\alpha m}\cdot \left(\frac{2}{\alpha^2}\delta^2n\right)^{\alpha m}n^{m-\alpha m}
\le \left[\left(\frac{e}{\alpha}\right)^{\alpha}\cdot\left(\frac{2}{\alpha^2}\delta^2\right)^{\alpha}\cdot n\right]^m\,.
$$
Dividing by $m!$ to get the number of unordered non-expanding
$m$-sets, and then multiplying by $p^m$ we get that the probability
that $R$ contains a non-expanding $m$-set is at most
\begin{equation}\label{lab1}
\frac{ \left[\left(\frac{e}{\alpha}\right)^{\alpha}\cdot\left(\frac{2}{\alpha^2}\delta^2\right)^{\alpha}\cdot np\right]^m}{m!}\le
\left[\left(\frac{2e\delta^2}{\alpha^3}\right)^{\alpha}\cdot\frac{enp}{m}\right]^m\,.
\end{equation}
Recall that we assumed $p\le\frac{2}{d}$ and $m\ge \frac{cn}{d}$, implying $\frac{np}{m}\le \frac{2}{c}$. Choosing $\delta>0$ from the lemma statement small enough guarantees that the expression in (\ref{lab1}) is, say, at most  $2^{-m}$. Applying the union bound over all possible values of $m$ establishes the lemma.
\end{proof}

\bigskip

\noindent{\bf Proof of Theorem \ref{th2}.} Set $\alpha_0=\frac{\epsilon}{25}$, $c=\epsilon$, and choose $\delta$ in the theorem statement to be $\delta(\alpha_0,c)$ from Lemma \ref{key_lemma}. Run the DFS algorithm of $G$ and feed it with a sequence $\bar{X}=(X_i)_{i=1}^n$ of i.i.d. Bernoulli($p$) random variables. Assume that $\bar{X}$ satisfies the properties stated in Lemma \ref{lem1}. We claim that after the first $\epsilon n$ vertex queries (of the type ``Whether $v\in R$?") of the DFS algorithm, the set $U$ contains at least $\frac{\epsilon^2 n}{5d}$ vertices, with the contents of $U$ forming
a path of desired length at that moment.  At that point, all $\epsilon n$ queried vertices reside in $S\cup U\cup W$, implying $|S\cup U\cup W|=\epsilon n$. Also, each positive answer to a query put a vertex in $U$ (that possibly has migrated further to $S$). Hence $|S\cup U|\ge\frac{\epsilon\left(1+\frac{3\epsilon}{4}\right)n}{d}$, by Property (4) of Lemma \ref{lem1}. Denote $|S|=m$. If $|U|\le\frac{\epsilon^2n}{5d}$, then it follows that
$$
m\ge\frac{\epsilon\left(1+\frac{3\epsilon}{4}\right)n}{d}-\frac{\epsilon^2n}{5d}=\frac{\left(\epsilon+\frac{11\epsilon^2}{20}\right)n}{d}\,.
$$
Also, by Property (3) of Lemma \ref{lem1}, $m\le \frac{2\epsilon n}{d}$. So Lemma \ref{key_lemma} is applicable, and we derive:
\begin{eqnarray*}
|N_G(S)|&\ge& (1-\alpha_0)\left(dm-\frac{d^2m^2}{2n}\right)\\
&\ge& (1-\alpha_0)\left(\epsilon+\frac{11\epsilon^2}{20}\right)\left(1-\frac{\epsilon}{2}-\frac{11\epsilon^2}{40}\right)n\\
&=&\left(1-\frac{\epsilon}{25}\right)\left(\epsilon +\frac{\epsilon^2}{20}-O(\epsilon^3)\right)n\\
&>&\epsilon n\,,
\end{eqnarray*}
for small enough $\epsilon>0$ --- a contradiction, since, as we stated earlier, at any point of the algorithm execution we have $N_G(S)\subseteq U\cup W$, and $|U\cup W|\le \epsilon n$.

\bigskip

\Remark Since by (\ref{eml1}) there is an edge between any two disjoint vertex subsets $B,C$ of an $(n,d,\lambda)$-graph $G$, as long as $|B|,|C|>\frac{\lambda n}{d}$, we obtain immediately that for $\lambda/d$ small enough as a function of $\epsilon$, the random set $R$ contains also a {\em cycle} of length proportional to $\frac{\epsilon^2n}{d}$. Indeed, take a path $P$ of length $\frac{\epsilon^2n}{5d}$, whose typical existence is guaranteed by Theorem \ref{th2}, and let $B,C$ be the first, resp. last, $\frac{\epsilon^2n}{15d}$ vertices of $P$. Then $G$ has an edge between $B$ and $C$, this edge obviously closes a cycle $C$ with the corresponding part of $P$; the length of this cycle is proportional to $\frac{\epsilon^2n}{d}$.

\medskip

\noindent{\bf Proof of Theorem \ref{th3}.} Set $\alpha_0=\frac{\epsilon}{5}$, $c=\epsilon^3$, and choose $\delta$ in the theorem statement to be $\delta(\alpha_0,c)$ from Lemma \ref{key_lemma}.
Run the DFS algorithm of $G$ and feed it with a random binary sequence $\bar{X}$. Assume that $\bar{X}$ satisfies the properties stated in Lemma \ref{lem1}. Let us focus on the situation after the first $\epsilon n$ vertex queries of the algorithm. We claim that at this moment we are in the midst of processing a connected component of $G[R]$ of size at least $\frac{\epsilon n}{d}$. Assume that at some moment $t\in[\epsilon^3n,\epsilon n]$ the set $U$ becomes empty. We have then: $|S\cup W|=t$ and $m:=|S|=\sum_{i=1}^t X_i\ge\frac{\left(1+\frac{3\epsilon}{4}\right)t}{d}$, by Property (4) of Lemma \ref{lem1}; also, $m\le \frac{2\epsilon n}{d}$,  by Property (3) of Lemma \ref{lem1}, allowing to apply Lemma \ref{key_lemma}. Since $U=\emptyset$, we have now: $N_G(S)\subseteq W$. Therefore
\begin{eqnarray*}
|W|&\ge& (1-\alpha_0)\left(dm-\frac{d^2m^2}{2n}\right)\\
&\ge& (1-\alpha_0)\left(1+\frac{3\epsilon}{4}\right)\left(1-\frac{\left(1+\frac{3\epsilon}{4}\right)t}{2n}\right)t\\
&\ge& \left(1-\frac{\epsilon}{5}\right)\left(1+\frac{3\epsilon}{4}\right)\left(1-\left(1+\frac{3\epsilon}{4}\right)\frac{\epsilon}{2}\right)t\\
&>&t\,,
\end{eqnarray*}
for small enough $\epsilon>0$ --- a contradiction. Hence $U$ never empties in the interval $[\epsilon^3n,\epsilon n]$. This means that all vertices added to $U$ during this period belong to the same connected component $C$, whose epoch contains this interval; their number is
$$
\sum_{i=\epsilon^3n}^{\epsilon n}X_i\ge \frac{\epsilon\left(1+\frac{3\epsilon}{4}\right)n}{d}-\frac{2\epsilon^3n}{d} \ge \frac{\epsilon n}{d}\
$$
(we used Property (2) of Lemma \ref{lem1} in the estimate above.)
It follows that under the above probabilistic assumptions  $G[R]$ has a component $C$ of at least $\frac{\epsilon n}{d}$ vertices, as claimed.

\section{Concluding remarks}
We have proven that in the site percolation model for an $(n,d,\lambda)$-graph $G$, under rather mild assumptions on the spectral ratio $\lambda/d$, the phase transition occurs at $p=\frac{1}{d}$: for $p=\frac{1-\epsilon}{d}$, \whp all connected components of the subgraph of $G$ induced by a random subset $R$ are of size at most logarithmic in $n$, while for $p=\frac{1+\epsilon}{d}$, the random set $R$ spans \whp a connected component of size at least $\frac{\epsilon n}{d}$ and a path of length proportional to $\frac{\epsilon^2 n}{d}$.

Although we have established the existence of the phase transition for this model of pseudo-random graphs in this paper, many further natural questions about site percolation on pseudo-random graphs have not been resolved here, and it would be nice to address them. Particular issues include, for the super-critical regime $p=\frac{1+\epsilon}{d}$:
\begin{itemize}
\item the uniqueness of the giant component, bounding sizes of all other components spanned by the random subset $R$;
\item accurate (in $\epsilon$) asymptotics of the size of the giant component in $G[R]$;
\item  upper bounding the length of a longest path/cycle spanned by $R$.
\end{itemize}
And of course, it would be very interesting to look into the critical regime $p=\frac{1+o(1)}{d}$, aiming to try and understand the continuous evolution of the size of the giant component spanned by $R$ from logarithmic to linear in $|R|$.

\bigskip

\noindent{\bf Acknowledgement.} The author would like to thank Amin Coja-Oghlan, Uri Feige, Daniel Reichman and Wojciech Samotij for stimulating discussions about this problem.

\end{document}